\documentclass[10pt,reqno]{amsart}
\usepackage{amssymb,amscd,url} 
\usepackage[all]{xy}  
\usepackage{color}

\theoremstyle{plain}
\newtheorem{theorem}{Theorem}[section]
\newtheorem{proposition}[theorem]{Proposition}

\newtheorem{lemma}[theorem]{Lemma}

\newtheorem{definition}[theorem]{Definition}
\newtheorem{remark}[theorem]{Remark}

\numberwithin{equation}{section}

\newcommand{\N}{\mathbb{N}}
\newcommand{\R}{\mathbb{R}}

\newcommand{\prob}{\mathbb{P}_N}
\newcommand{\qrob}{\mathbb{Q}_N}
\newcommand{\probsp}{\mathcal{P}([0,1])}
\newcommand{\veff}{V_{\mathrm{eff}}}

\begin{document}

\title{A remark on Jacobi ensemble} 
\author[I.\ Ozeki]{Ikuya OZEKI}
\address{
Graduate School of Mathematics, 
Nagoya University, 
Furocho, Chikusaku, Nagoya, 464-8602, Japan
}
\email{ikuyaoze@gmail.com}
\thanks{}

\date{\today}
\maketitle
\begin{abstract}
We prove the large deviation principle for the supports of Jacobi ensembles following Guionnet's method. 
\end{abstract}

\allowdisplaybreaks{

\section{Introduction} 

Let $P(N), Q(N)$ be a sequence of two random projection matrices. Its statistical behavior can be understood by means of $P(N)Q(N)P(N)$, whose eigenvalue distribution is known to be a Jacobi ensemble in a natural setup. Hiai--Petz \cite{Hiai-Petz} proved the large deviation principle for empirical probability distributions of Jacobi ensembles. See \cite{Hiai-Petz} for further references on these facts.  

In this short note, we will explain that the method of \cite[Theorem 4.8]{Guionnet} (related results are summarized therein) due to Guionnet that establishes the large deviation principle for the supports of $\beta$-ensembles in the large $N$ limit certainly works well for the supports of ($\beta$-)Jacobi ensembles too.  

\section{Jacobi ensemble} 

Let $\probsp$ be the set of all probability measures on $[0,1]$ with a metric $d$ inducing the weak topology. 

\begin{definition}\label{Jacobi-ensemble}
For each $N\in \N$ and $n(N)\in \mathbb{N}, \kappa(N), \lambda(N)\in [0,\infty)$, let $\mathbb{P}_N = \mathbb{P}_N^{(n(N),\kappa(N),\lambda(N))}$ be the probability measure on $[0,1]^{n(N)}$ given by
\begin{equation}\label{eigenvalues-distribution}
\frac{1}{Z(N)}\exp\Big(-2n(N)\sum_{i=1}^{n(N)}V_N(x_i)+2\sum_{1\leq i<j \leq n(N)}\log|x_i-x_j|\Big)\prod_{i=1}^{n(N)}1_{[0,1]}(x_i)dx_i 
\end{equation}
with normalization constant $Z(N)=Z(N;n(N),\kappa(N),\lambda(N))$, where we define 
\[
V_N(x) = V_N^{(n(N),\kappa(N),\lambda(N))}:= -\frac{\kappa(N)}{2n(N)}\log x-\frac{\lambda(N)}{2n(N)}\log(1-x). 
\] 
\end{definition}

\begin{remark}{\textup($\beta$-Jacobi ensemble\textup)}
A $\beta$-Jacobi ensemble is a probability distribution over $[0, 1]^N$ whose density function \textup(with respect to the Lebesgue measure\textup) is proportional to
\[\prod_{i = 1}^N x_i^{a(N)}(1-x_i)^{b(N)}\prod_{1 \leq i < j \leq N}|x_i-x_j|^\beta. \]
This measure is a special case of Definition \ref{Jacobi-ensemble} when $n(N) = N$, $\kappa(N) = 2a(N)/\beta$ and $\lambda(N) = 2b(N)/\beta$. 
\end{remark}

Throughout this note, we assume that 
\begin{equation}\label{assumption}
n(N)/N\to \rho, \quad \kappa(N)/N\to \kappa, \quad \lambda(N)/N\to \lambda \qquad \text{(as $N\to \infty$)} 
\end{equation}
for some $\rho \in (0,\infty)$ and $\kappa, \lambda \in [0,\infty)$. Let us recall Hiai--Petz's result \cite[Proposition 2.1]{Hiai-Petz} on LDP for the sequence $\mathbb{P}_N$.

\begin{proposition}\label{H-P} The following hold{\rm:} 
\begin{itemize} 
\item[(i)] The limit $\lim_{N\to \infty}N^{-2}\log Z(N)$ exists and equals $\rho^2B(\kappa/\rho, \lambda/\rho)$ with the function $B(s, t)$ in \cite[proposition 2.2]{Hiai-Petz}
\item[(ii)] When $(x_1, \ldots, x_{n(N)})$ is distributed under $\prob$, the empirical probability measure
\[
\hat{\mu}_{n(N)} := \frac{1}{n(N)}\sum_{i=1}^{n(N)} \delta_{x_i}
\]
satisfies the large deviation principle in scale $1/N^2$ with good rate function
\begin{align*} 
I(\mu) 
&:= -\rho^2\iint \log|x-y|d\mu(x)d\mu(y) \\
&\qquad -\rho\int_0^1\kappa \log x+ \lambda \log(1-x)d\mu(x)+\rho^2B(\kappa/\rho, \lambda/\rho)
\end{align*}
for $\mu \in \probsp$. 
Moreover, there exists a unique minimizer $\mu_0\in \probsp$ of $I(\mu)$ with $I(\mu_0) = 0$. 
\end{itemize}
\end{proposition}

\section{Main result} 

Define the effective potential
\[\veff(x) := V(x)-\int_0^1 \log|x-y|d\mu_0(y)-D_{\rho, \kappa, \lambda}, \]
where 
\[V(x) := -\frac{\kappa}{2\rho}\log x-\frac{\lambda}{2\rho}\log(1-x), \]
and
\[D_{\rho, \kappa, \lambda} := -B(\kappa/\rho, \lambda/\rho)-\int_0^1 V(x)d\mu_0(x). \]

Here are two lemmas. 

\begin{lemma}\label{effective-potential-property}
The effective potential $\veff$ satisfies the condition
\begin{align*}
\veff(x)
\begin{cases}
= 0 & {\text{quasi-everywhere on $\mathrm{supp}(\mu_0)$}}, \\
\geq 0 & {\text{if $x\in [0,1]\setminus\mathrm{supp}(\mu_0)$},} 
\end{cases}
\end{align*}
where $\mathrm{supp}(\mu_0)$ denotes the support of $\mu_0$. 
\end{lemma}
\begin{proof} This immediately follows from \cite[Theorem I.1.3]{Saff-Totik}. (See the proof of \cite[Proposition 2.1]{Hiai-Petz} too.)
\end{proof}
\newpage
\begin{lemma}\label{exp-equivalence}
The probability measure $\qrob := \prob^{(n(N)-1,\kappa(N),\lambda(N))}$ on $[0,1]^{n(N)-1}$ is exactly 
\begin{align*}
\frac{1}{C(N)} &\exp\Big(-2n(N)\sum_{i=1}^{n(N)-1}V_N^{(n(N),\kappa(N),\lambda(N))}(x_i)\\
&\qquad+2\sum_{1\leq i<j \leq n(N)-1}\log|x_i-x_j|\Big) \prod_{i=1}^{n(N)-1}1_{[0,1]}(x_i)dx_i  
\end{align*}
with $C(N) = Z(N;n(N)-1,\kappa(N),\lambda(N))$. 
\end{lemma}
\begin{proof}
This follows from $V_N^{(n(N),\kappa(N),\lambda(N))}(x)=\frac{n(N)-1}{n(N)}V_N^{(n(N)-1,\kappa(N),\lambda(N))}(x)$. 
\end{proof}

Let us prove our main result. 

\begin{theorem}
Define the probability measure $\widehat{\prob}$ on $[0,1]$ by
\[
\widehat{\prob}(X) := \prob\big(\{(x_1, \ldots, x_{n(N)}) \in [0,1]^{n(N)}; \{x_1,\dots,x_{n(N)}\} \cap X \neq \emptyset \}\big)
\]
for any Borel subset $X$ of $[0,1]$. Then, the sequence $\widehat{\prob}$ satisfies the large deviation principle in scale $1/N$ with good rate function $2\rho \veff$, that is, for any open set $O$ and any closed set $F$ of $[0,1]$, 
\begin{align} 
&\liminf_{N\to \infty} \frac{1}{N}\log \widehat{\prob}(O) \geq -\inf_{x \in O}2\rho \veff(x), \label{true-LDP-lower-bound}\\
&\limsup_{N\to \infty} \frac{1}{N}\log \widehat{\prob}(F) \leq -\inf_{x \in F}2\rho \veff(x). \label{true-LDP-upper-bound}
\end{align} 
\end{theorem}

Lemma \ref{effective-potential-property} shows that $\inf_{x \in O}\veff(x) = 0$ if $O\cap \mathrm{supp}(\mu_0) \neq \emptyset$ for any open subset $O$ of $[0,1]$. Thus, the ``state space" of this large deviation principle should sit inside $[0,1]\backslash \mathrm{supp}(\mu_0)$ rather than $[0,1]$. 

\medskip
The proof below essentially follows the idea in \cite[Theorem 4.8]{Guionnet}, but we have to take care of some details (for example, the part of proving \eqref{LDP-upper-bound} follows \cite[Lemma 2.6.7]{Anderson-Guionnet-Zeitouni} instead). 

\begin{proof}
The proof is divided into four steps. 

\medskip\noindent
\underline{Step 1. $\veff$ is a good rate function:} 
Note that the effective potential $\veff$ is lower semicontinuous on $[0,1]$. Indeed, $\veff(x) = \sup_{M>0}\veff^M(x)$ where 
\[
\veff^M(x) := (V(x)\land M)-\int_0^1\log(|x-y|\lor M^{-1})d\mu_0(y)-D_{\rho, \kappa, \lambda} 
\]
for any $M>0$, which is clearly continuous on $[0,1]$. Thus, the effective potential $\veff(x)$ is lower semicontinuous. It immediately follows that $\{x \in [0,1]; \veff(x)\leq K\}$ is closed and hence compact for any $K\in \R$. 

\bigskip
Define
\[
\gamma_N(X) := \qrob\Big[\int_X \exp\Big(-2n(N)V_N(\xi)+2(n(N)-1)\int_0^1\log|\xi-\eta|d\hat{\mu}_{n(N)-1}(\eta)\Big)d\xi\Big],  
\]
where $\qrob[\,\cdot\,]$ denotes the expectation with respect to $\qrob$ (see Lemma \ref{exp-equivalence}). 

\bigskip\noindent
\underline{Step 2. A large deviation lower bound for $\gamma_N$:} 
Let $O$ be an open subset of $[0,1]$. Since $O \cap (0,1)$ is open in $\mathbb{R}$, we choose and fix an arbitrary $x \in O \cap (0,1)$ so that $[x-\delta, x+\delta] \subset O \cap (0,1)$ for all sufficiently small $\delta>0$. We fix such a $\delta>0$ for a while. In what follows, we write $V_N = V_N^{(n(N),\kappa(N),\lambda(N))}$ for short as in \S2. 

We have 
\begin{align*}
&\gamma_N\Big(O\cap (0,1)\Big)\geq \\
&\qquad\qrob\Big[\int_{x-\delta}^{x+\delta} \exp\Big(-2n(N)V_N(\xi)+2(n(N)-1)\int_0^1\log|\xi-\eta|d\hat{\mu}_{n(N)-1}(\eta)\Big)d\xi\Big] 
\end{align*}
with $\hat{\mu}_{n(N)-1} = (n(N)-1)^{-1}\sum_{i=1}^{n(N)-1}\delta_{x_i}$. Since $V_N(\xi)$ converges to $V(\xi)$ uniformly on $[x-\delta, x+\delta]$, for any $\varepsilon>0$, there exists an $N_0\in \N$, such that $N\geq N_0$ implies $V_N(\xi)\leq V(\xi)+\varepsilon$ for all $\xi \in [x-\delta, x+\delta]$. 
Hence 
\begin{align*}
&\gamma_N\Big(O\cap (0,1)\Big) \geq \\
&\quad\qrob\Bigg[\int_{x-\delta}^{x+\delta}\exp\Big(-2n(N)(V(\xi)+\varepsilon)+2(n(N)-1)\int_0^1\log|\xi-\eta|d\hat{\mu}_{n(N)-1}(\eta)\Big)d\xi\Bigg]. 
\end{align*} 
With $E_\delta(V):= \sup\{|V(x)-V(y)|; |x-y|<\delta\}$ ({\it n.b.}$x$ has been fixed), it follows that
\begin{align*}
&\gamma_N\Big(O\cap (0,1)\Big) 
\geq \\
&\qquad 
e^{-2n(N)(V(x)+E_\delta(V)+\varepsilon)}
\qrob\Bigg[\int_{x-\delta}^{x+\delta} \exp\Big(2(n(N)-1)\int_0^1\log|\xi-\eta|d\hat{\mu}_{n(N)-1}(\eta)\Big)d\xi\Bigg]. 
\end{align*}
Using Jensen's inequality, we have 
\begin{align*}
&\gamma_N\Big(O\cap (0,1)\Big) \geq \\
&\quad 2\delta\,\exp\Bigg(-2n(N)(V(x)+E_\delta(V)+\varepsilon)+2(n(N)-1)\qrob\Big[\int_0^1 H_{x,\delta}(\eta)d\hat{\mu}_{n(N)-1}(\eta)\Big]\Bigg),  
\end{align*}
where 
\begin{align*}
&H_{x, \delta}(\eta):=\int_{x-\delta}^{x+\delta}\log|\xi-\eta|\frac{d\xi}{2\delta}. 
\end{align*}
Note that $H_{x,\delta}(\eta)$ can explicitly be calculated and turns out be a continuous function (in $\eta$) over $\mathbb{R}$. 
Thanks to Lemma \ref{exp-equivalence} and Theorem \ref{H-P} ({\it n.b.}$(n(N)-1)/N \to \rho$ as $N\to\infty$), we see that $\hat{\mu}_{n(N)-1}$ weakly converges to $\mu_0$ almost surely with $\qrob$, and hence
\[
R(\delta, N) := \int_0^1 H_{x,\delta}(\eta)d\hat{\mu}_{n(N)-1}(\eta) - \int_0^1 H_{x,\delta}(\eta)d\mu_0(\eta) \to 0
\]
as $N\to\infty$ almost surely with $\qrob$. 

Observe that 
\begin{align*}
\gamma_N\Big(O\cap (0,1)\Big) \geq 2\delta\,\exp\Big(&-2n(N)(V(x)+E_\delta(V)+\varepsilon)\\
&+2(n(N)-1)\int_0^1 H_{x, \delta}(\eta)d\mu_0(\eta)+2(n(N)-1)R(\delta, N)\Big),  
\end{align*}
and hence
\begin{equation}\label{Eq3.9}
\liminf_{N\to \infty}\frac{1}{N}\log \gamma_N\Big(O\cap (0,1)\Big) \geq -2\rho(V(x)+E_\delta(V)+\varepsilon)+2\rho\int_0^1 H_{x, \delta}(\eta)d\mu_0(\eta)
\end{equation}
holds by \eqref{assumption}. 

Write $F_\eta(x):=\int_0^x \log|\xi-\eta|d\xi$. We observe that 
\begin{align*}
\lim_{\delta \to 0}\int_0^1 H_{x, \delta}(\eta) d\mu_0(\eta)& = \lim_{\delta \to 0}\int_0^1 \frac{F_\eta(x+\delta)-F_\eta(x-\delta)}{2\delta}d\mu_0(\eta)\\
& = \lim_{\delta \to 0}\int_0^1 \frac{F_\eta(x+\delta)-F_\eta(x)}{2\delta}+\frac{F_\eta(x-\delta)-F_\eta(x)}{2(-\delta)}d\mu_0(\eta)\\
& = \int_0^1 \lim_{\delta \to 0} \frac{F_\eta(x+\delta)-F_\eta(x)}{2\delta}+\frac{F_\eta(x-\delta)-F_\eta(x)}{2(-\delta)}d\mu_0(\eta)\\
& = \int_0^1 \log|x-\eta|d\mu_0(\eta), 
\end{align*}
where we used the dominated convergence theorem in the third line and the  fundamental theorem of calculus in the fourth line. Therefore, taking the limit of \eqref{Eq3.9} as $\delta \to 0$, we have
\[
\liminf_{N\to \infty}\frac{1}{N}\log \gamma_N\Big(O\cap (0,1)\Big) \geq -2\rho\Big(V(x)-\int_0^1 \log|x-y|d\mu_0(y)\Big)-2\rho\, \varepsilon. 
\]
Since $\varepsilon>0$ can arbitrary be small, we have, for any $x \in X \cap (0,1)$, 
\[
\liminf_{N\to \infty}\frac{1}{N}\log \gamma_N\Big(O\cap (0,1)\Big) \geq -2\rho\Big(V(x)-\int_0^1 \log|x-y|d\mu_0(y)\Big). 
\]
Since $V(x)-\int_0^1 \log|x-y|d\mu_0(y)=\infty$ if $x \in\{0, 1\}$, we conclude that
\begin{equation}\label{LDP-lower-bound}
\liminf_{N\to \infty}\frac{1}{N}\log \gamma_N(O)\geq -2\rho \inf_{x\in X}\Big(V(x)-\int_0^1 \log|x-y|d\mu_0(y)\Big). 
\end{equation}

\bigskip\noindent
\underline{Step 3. A large deviation upper bound for $\gamma_N$:} Let $F$ be closed subset of $[0,1]$ and define
\[
\Phi_N^L(\xi,\mu) := V_N(\xi)\wedge L-\frac{n(N)-1}{n(N)}\int_0^1\log(|\xi-\eta|\vee L^{-1}) d\mu(\eta)
\]
on $[0,1]\times \probsp$ for any $L > 1$. Then we have 
\begin{equation}\label{Eq3.7} 
-\Phi_N^L \leq -(V_N\land L) \leq C 
\end{equation}
for some $C \in \mathbb{R}$, since $\log(|\xi-\eta|\lor L^{-1}) \leq 0$ on $(\xi, \eta)\in [0,1]\times[0,1]$. 

Observe that 
\[
\gamma_N(F) \leq \qrob\Big[\int_F \exp\Big(-2n(N)\Phi_N^L(\xi, \hat{\mu}_{n(N)-1})\Big)d\xi\Big]. 
\]
Choose and fix an arbitrarily small $\delta>0$. Dividing the integration range of $\qrob$ into two parts $\{\hat{\mu}_{n(N)-1}\in \{d(\cdot, \mu_0) \leq \delta\}\}$ and $\{\hat{\mu}_{n(N)-1}\in \{d(\cdot, \mu_0) > \delta\}\}$ and using \eqref{Eq3.7}, we obtain that  
\begin{align*}
\gamma_N(F)
&\leq \exp\Big(-2n(N)\inf_{(\xi, \mu)\in F\times\{d(\cdot, \mu_0) \leq \delta\}}\Phi_N^L(\xi, \mu)\Big) \\
&\qquad\qquad +e^{2n(N)C}\qrob\Big(\hat{\mu}_{n(N)-1}\in \{d(\cdot, \mu_0) > \delta\}\Big), 
\end{align*}
and hence  
\begin{align*}
\limsup_{N\to \infty}\frac{1}{N}\log\gamma_N(F) \leq \max\Big\{&\limsup_{N\to \infty}-\frac{2n(N)}{N}\inf_{(\xi, \mu)\in F\times\{d(\cdot, \mu_0) \leq \delta\}}\Phi_N^L(\xi, \mu),\\
&\limsup_{N\to \infty}\frac{2n(N)C}{N}+\limsup_{N\to \infty}\frac{1}{N}\log\qrob\Big(\hat{\mu}_{n(N)-1}\in \{d(\cdot, \mu_0) > \delta\}\Big)\Big\}
\end{align*}
by \eqref{assumption}. By Theorem \ref{H-P} together with Lemma \ref{exp-equivalence} we have 
\[
\limsup_{N\to \infty}\frac{1}{N}\log\qrob\Big(\hat{\mu}_{n(N)-1}\in \{d(\cdot, \mu_0) > \delta\}\Big)=-\infty, 
\]
and hence 
\[
\limsup_{N\to \infty}\frac{1}{N}\log\gamma_N(F) \leq -2\rho\liminf_{N\to \infty} \inf_{(\xi, \mu)\in F\times\{d(\cdot, \mu_0) \leq \delta\}}\Phi_N^L(\xi, \mu). 
\]
By \eqref{assumption}, $\Phi_N^L(\xi, \mu)$ converges to $V(\xi)\land L -\int_0^1\log(|\xi-\eta|\lor L^{-1})d\mu(\eta)$ uniformly on $[0,1]$ as $N\to \infty$, and thus
\[
\limsup_{N\to \infty}\frac{1}{N}\log\gamma_N(F) \leq -2\rho\inf_{(\xi, \mu)\in F\times\{d(\cdot, \mu_0) \leq \delta\}}\Big(V(\xi)\land L-\int_0^1\log(|\xi-\eta|\lor L^{-1})d\mu(\eta)\Big). 
\]
Since $V(\xi)\land L-\int_0^1\log(|\xi-\eta|\lor L^{-1})d\mu(\eta)$ is continuous on $[0,1]\times \probsp$, we observe that
\begin{align*}
&\inf_{(\xi, \mu)\in F\times\{d(\cdot, \mu_0) \leq \delta\}}\Big(V(\xi)\land L-\int_0^1\log(|\xi-\eta|\lor L^{-1})d\mu(\eta)\Big)\\
&\qquad\qquad\qquad \longrightarrow \inf_{\xi \in F}\Big(V(\xi)\land L-\int_0^1\log(|\xi-\eta|\lor L^{-1})d\mu_0(\eta)\Big)
\end{align*}
as $\delta\to0$ (see the proof of \cite[Lemma 4.1.6(a)]{Dembo-Zeitouni}). 
Thus, 
\[
\limsup_{N\to \infty}\frac{1}{N}\log\gamma_N(F) \leq -2\rho \inf_{\xi \in F}\Big(V(\xi)\land L-\int_0^1\log(|\xi-\eta|\lor L^{-1})d\mu_0(\eta)\Big), 
\]
and letting $L\to \infty$ we conclude that 
\begin{equation}\label{LDP-upper-bound}
\limsup_{N\to \infty}\frac{1}{N}\log\gamma_N(F) \leq -2\rho \inf_{\xi \in F}\Big(V(\xi)-\int_0^1\log|\xi-\eta|d\mu_0(\eta)\Big). 
\end{equation} 

\bigskip\noindent
\underline{Step 4. Transition from $\gamma_N$ to $\widehat{\prob}$:}  
We first claim that 
\begin{equation}\label{gamma-evaluation}
\frac{\gamma_N(X)}{\gamma_N([0,1])} \leq \widehat{\prob}(X)\leq n(N)\frac{\gamma_N(X)}{\gamma_N([0,1])}. 
\end{equation}
This follows from the following two observations: 
\begin{align*}
&\prob(\{(x_1, \ldots, x_{n(N)}) \in [0,1]^{n(N)}; \ x_{n(N)} \in X)\}) \\
&\qquad\leq \widehat{\prob}(X) 
\leq n(N)\prob(\{(x_1, \ldots, x_{n(N)}) \in [0,1]; x_{n(N)} \in X)\})
\end{align*}
and 
\begin{align*}
&\prob(\{(x_1, \ldots, x_{n(N)}) \in [0,1]^{n(N)}; \ x_{n(N)} \in X)\})\\
&=\frac{1}{Z(N)}\int_{[0,1]^{n(N)-1}}\int_X \exp\Big(-2n(N)\sum_{i=1}^{n(N)} V_N(x_i)+2\sum_{1\leq i<j \leq n(N)}\log|x_i-x_j|\Big)\prod_{i=1}^{n(N)}dx_i\\
&=\frac{C(N)}{Z(N)}\frac{1}{C(N)}
\int_{[0,1]^{n(N)-1}}\left(\int_X \exp\Big(-2n(N)V_N(\xi)+2\sum_{i=1}^{n(N)-1}\log|x_i-\xi|\Big)d\xi\right)\\
&\hspace{2cm}\times \exp\Big(-2n(N)\sum_{i=1}^{n(N)-1}V_N(x_i)+2\sum_{1\leq i<j \leq n(N)-1}\log|x_i-x_j|\Big)\prod_{i=1}^{n(N)-1} dx_i \\
&=\frac{\gamma_N(X)}{\gamma_N([0,1])}
\end{align*}
by $\gamma_N([0,1])=Z(N)/C(N)$. 

By \eqref{LDP-lower-bound}, \eqref{LDP-upper-bound} and Lemma \ref{effective-potential-property}, we have
\[
\lim_{N\to\infty}\frac{1}{N}\log\gamma_N([0,1]) = 
-2\rho \inf_{\xi \in [0,1]}\Big(V(\xi) - \int_0^1\log|\xi-\eta|d\mu_0(\eta)\Big)
=-2\rho D_{\rho, \kappa, \lambda}. 
\]
This and $\lim_{N\to\infty}\frac{1}{N}\log n(N) = 0$ (thanks to \eqref{assumption}) enable us to derive \eqref{true-LDP-lower-bound} and \eqref{true-LDP-upper-bound} from \eqref{LDP-lower-bound} and \eqref{LDP-upper-bound}, respectively. Hence we have completed the proof.  
\end{proof}

\section*{ACKNOWLEDGEMENT}
The author thanks his supervisor, Professor Yoshimichi Ueda for conversations, comments and editorial supports to this note.

\end{document}